\documentclass[11 pt,reqno]{amsart}
\usepackage{amsmath}
\usepackage{amssymb}
\usepackage{amsthm}
\usepackage{amsfonts}
\usepackage{graphicx}
\usepackage{graphics}
\usepackage{epsfig}
\usepackage{color}

 \newtheorem{thm}{Theorem}[section]

 \newtheorem{prop}[thm]{Proposition}
 \theoremstyle{definition}
 \newtheorem{defn}[thm]{Definition}
 \newtheorem{conj}[thm]{Conjecture}
 \newtheorem{prob}[thm]{Problem}
 \theoremstyle{remark}

\numberwithin{equation}{section} \numberwithin{figure}{section}

\newcommand{\C}{{\mathbb C}}
\newcommand{\D}{{\mathbb D}}
\newcommand{\T}{{\mathbb T}}

\newcommand{\R}{{\mathbb R}}
\newcommand{\N}{{\mathbb N}}

\DeclareMathOperator{\Span}{span}

\begin{document}
\bibliographystyle{alpha}

\title[Optimal approximants]{Some problems on optimal approximants}
\author[Seco]{Daniel Seco}
\address{Departament de Matem\`atica Aplicada i An\`alisi, Facultat de Matem\`atiques, Universitat de Barcelona, Gran Via 585, 08007 Barcelona, Spain.} \email{dseco@mat.uab.cat}
\thanks{The author was supported by ERC Grant Agreement n.291497 of the EU's Seventh Framework Programme (FP/2007-2013), and by Ministerio de Econom\'ia y Competitividad Projects MTM2011-24606 and MTM2014-51824-P}
\date{\today}

\begin{abstract}
We present an account of different problems that arise in relation
with cyclicity problems in Dirichlet-type spaces, in particular with
polynomials $p$ that minimize the norm $\|pf-1\|$.
\end{abstract}

\maketitle


\section{Introduction}

Denote by $\D$, the unit disk of the complex plane and let $\alpha
\in \R$. The \emph{Dirichlet-type space}, $D_{\alpha}$, is the space
of analytic functions $f: \D \rightarrow \C$, given by a Taylor
expansion $f(z) = \sum_{k=0}^{\infty} a_k z^k$, for which
\begin{equation}\label{eqn1} \|f\|^2_{\alpha} = \sum_{k=0}^{\infty}
|a_k|^2(k+1)^{\alpha} < \infty.
\end{equation}
This one-parameter family joins together three classical examples:
Bergman ($\alpha=-1$), Hardy ($\alpha=0$) and Dirichlet ($\alpha=1$)
spaces. Since $\alpha$ will be fixed on each problem, whenever there
is no ambiguity, we will use the notation $\|\cdot\|$ and $\left<,
\right>$ for the norm and inner product in $D_{\alpha}$. All the
$D_{\alpha}$ spaces are reproducing kernel Hilbert spaces, meaning
that for each $\omega \in \D$ there exists a function $k_{\omega,
\alpha}$ such that for any $f \in D_{\alpha}$ we have
\begin{equation}\label{eqn3}\left< f, k_{\omega,\alpha} \right> = f(\omega).\end{equation} They are obviously
nested, in the sense that $D_{\alpha} \subset D_{\beta}$ for $\alpha
> \beta$, and the differentiation operator sends $D_{\alpha}$ to
$D_{\alpha-2}$ in such a way that $\|f\|^2_{\alpha} \approx |f(0)|^2
+ \|f'\|^2_{\alpha -2}$. When $\alpha <2$, an equivalent norm to
that in \eqref{eqn1} is
\begin{equation}\label{eqn2} \|f\|^2_{\alpha,*}= |f(0)|^2 +
\int_{\D} |f'(z)|^2 (1-|z|^2)^{1-\alpha} dA(z),\end{equation} where
$dA(z)= \frac{dxdy}{\pi}$. For $\alpha > 1$, the spaces are closed
under multiplication and the functions in these spaces are
continuous to the boundary, whereas this is not true for any space
for which $\alpha \leq 1$, giving the Dirichlet space a critical
situation for many problems. However, the operator $M_p$, of
multiplication by an analytic polynomial $p$, is always a bounded
operator in $D_{\alpha}$. For more details, we refer the reader to
\cite{HKZ, Duren, Garnett, EFKMR}.

On this work, we will concentrate on problems related to the
(forward) \emph{shift operator} $S: D_{\alpha} \rightarrow
D_{\alpha}$, given by
\begin{equation}\label{eqn4} Sf(z) =zf(z). \end{equation} In terms
of multiplication operators, $S=M_z$ but in terms of Taylor series,
the operator shifts each coefficient to the next position (and hence
its name). It is a deeply studied operator (see, for instance,
\cite{Niko}). A function $f \in D_{\alpha}$ is called \emph{cyclic}
(in $D_{\alpha}$ and with respect to the shift operator) if the
polynomial multiples of $f$ span a dense subspace. Denote by $[f]$
the smallest closed subspace of $D_{\alpha}$ which is invariant
under the shift and contains $f$. Then cyclic functions are exactly
those for which $[f] = D_{\alpha}$. From this definition, and since
polynomials are themselves dense in $D_{\alpha}$, it becomes clear
that $g \equiv 1$ is always a cyclic function and that a necessary
and sufficient condition for cyclicity of $f$ is that there exists a
cyclic function $g \in [f]$. It is therefore sufficient to study
whether there exist some family of polynomials $\{p_n\}_{n \in \N}$
such that
\begin{equation}\label{eqn5}\|p_n f -1\|^2 \stackrel{n \rightarrow \infty}{\rightarrow} 0.\end{equation}
From now on, we will denote by $\mathcal{P}_n$ the space of
polynomials of degree less or equal to $n$.
\begin{defn}\label{defn1} A polynomial $p$ that minimizes the norm $\|pf-1\|^2$
among those $p \in \mathcal{P}_n$ is called an \emph{optimal
approximant} (to $1/f$ of degree $n$ in $D_{\alpha}$).\end{defn} The
objective of this work is to present several problems related with
optimal approximants in this sense. It is clear that their behavior
characterizes cyclic functions, and they were introduced in
\cite{BCLSS13}.

The boundedness of point evaluations (a direct consequence of the
reproducing kernel property) guarantees that a function $f$ is not
cyclic if $f(\omega)=0$ at some point $\omega \in \D$, since
$\|pf-1\|$ will be controlled from below by a constant depending on
$\omega$ times $|p(\omega)f(\omega)-1|$. On the other hand, a
function which is holomorphic on a disk of radius larger than 1,
without zeros on the closed disk will be cyclic in any $D_{\alpha}$
since the Taylor polynomials of $1/f$ will make the norm in
\eqref{eqn5} tend to 0. Hence, the behavior towards the boundary of
the disk is often relevant in the study of cyclicity. Also, from
\eqref{eqn5} and the definitions of the norms it is clear that there
exists a hierarchy with respect to the cyclicity in Dirichlet-type
spaces, i.e., for $\alpha > \beta$, cyclic in $D_{\alpha}$ implies
cyclic in $D_{\beta}$.

In any case, a characterization for cyclicity in general is not
available although several steps have been taken in this direction.
A celebrated theorem of Beurling states that a function is cyclic in
the Hardy space if and only if it is an outer function. Outer
functions are those for which the logarithm satisfies a Mean Value
Property over the unit circle. A previous result by Smirnov and
reproved by Beurling allows to factor $f \in D_0$ as the product of
an outer $\theta$ and an inner function $I$, that is, a holomorphic
function such that $|I(z)| \leq 1 \forall z \in \D$ and
$|I(e^{i\theta})|=1$ for almost every $\theta \in [0, 2\pi)$. In
this sense, Beurling showed that a function is cyclic in Hardy if
and only if its inner factor is trivial.

In general Dirichlet-type spaces, the question was studied in detail
by Brown and Shields in \cite{BS84}. They showed that in the case
when the space is an algebra ($\alpha > 1$), a function $f$ is
cyclic if and only if $1/f$ is in the space $H^{\infty}$ of bounded
analytic functions, that is, if $f$ has no zeros in the closed unit
disk. Brown and Shields also showed that cyclic functions in the
Dirichlet space are different from those in Hardy and different from
those in $D_\alpha$ for $\alpha
>1$. They are different from the Hardy space case in that their
zero sets on the boundary need to be small (sets of zero logarithmic
capacity), and different from $\alpha >1$ since polynomials with
zeros on the boundary are cyclic in Dirichlet (for example, the
function $f(z)=1-z$). For information on logarithmic capacity, we
refer the reader to \cite{EFKMR,Garnett,SaTo}. In their article, the
authors proposed the following question:
\begin{conj}[Brown-Shields]\label{conj1} A function $f \in D_1$ is cyclic in $D_1$ if and only if it is outer and it has a set of boundary zeros of logarithmic capacity equal to zero.\end{conj} This problem stands open today and it has
attracted attention from a diversity of authors. Several of the
problems we will present are motivated by this conjecture.

Several attempts at solving Conjecture \ref{conj1} have dealt with
the concept of Bergman-Smirnov Exceptional (BSE) sets, a family of
subsets of $\T$ for which the relative part of Conjecture
\ref{conj1} holds and that includes all countable and many
uncountable closed sets. An excellent reference for this topic is
the book \cite{EFKMR} and we will not deal with it in here, since
the sets for which the BSE condition is not known have complicated
expressions and therefore, it seems out of the question by now to
work with the optimal approximants for functions that would be of
interest in this context.

\section{Polynomial proofs of cyclicity theorems}

\subsection{The classical theorems} The concept of optimal approximant as in Definition
\ref{defn1} was first introduced in \cite{BCLSS13}. See also
\cite{FrMaSe}. There it was shown that for each $\alpha \in \R$, $n
\in \N$, and $f \in D_{\alpha}$ (not identically zero) there exists
a unique optimal approximant to $1/f$ of degree $n$ in $D_\alpha$,
$p^*_n$. Moreover, the authors showed that the coefficients
$(c_k)_{k=0}^n$ of $p^*_n$ are given as the only solution to a
linear system, of the form
\begin{equation}\label{eqn6}
Mc=b
\end{equation}
where $M=(\left<z^j f, z^k f\right>)_{j,k=0}^n$, and $b=(\left<1,z^j
f\right>)_{j=0}^n$.

In the particular case of the Hardy space, it is easy to obtain, as
a corollary, a new proof of part of Beurling's Theorem, namely, that
there can only be one cyclic function in $D_0$ with a given outer
part: Indeed, let $f\in D_0$. From Parseval identity, the elements
of the matrix $M$, $M_{j,k}= \left<z^j f, z^k f\right>$ are equal to
$\lim_{r \rightarrow 1} \int_0^{2\pi} r^{j-k}e^{(j-k)i\theta} |f
(re^{i \theta})|^2 d\theta$, which does not depend on the inner part
of the function $f$.

In fact, it is easy to see that the property $<z^j f,f>=0$ for all
$j \geq 1$ characterizes inner functions among Hardy space functions
(not identically zero). From now on, we will refer to $<z^j f, z^k>$
as the \emph{moments} of $f$. f The simplicity of the proof
encourages one to wonder whether it is possible to complete a proof
of Beurling's Theorem in this way:

\begin{prob}
Show that outer functions are cyclic in $D_0$, using only
information on the optimal approximants and the moments of $f$.
\end{prob}

Two different and simple characterizations of cyclicity in Hardy
that use only information on the approximants were already given in
\cite{BKLSS4}. The proof of existence and uniqueness of $p^*_n$ lies
on the fact that $V_n = \mathcal{P}_n f$ is a finite dimensional
subspace of a Hilbert space, and therefore there is a unique
orthogonal projection. In fact, $p^*_n f$ will be the orthogonal
projection of $1$ onto $V_n$ and this tells us that $\|p^*_n f
-1\|^2 = 1-p^*_n(0)f(0)$. In particular, a function is cyclic if and
only if $p^*_n(0)$ tends to $1/f(0)$ as $n$ tends to $\infty$. This
holds in any of the $D_{\alpha}$ spaces. From now on, $Z(f)$ will
denote the zero set of a function $f$. The following was recently
shown in \cite{BKLSS4}: \begin{thm} Let $f \in D_0$. Then $f$ is
cyclic if and only if
 \begin{equation}\label{eqn10} \prod_{n=0}^{\infty} \left(1-\prod_{z_k \in
Z(p^*_n)} |z_k|^{-2} \right) =
\frac{\overline{f(0)}}{\|f\|^2},\end{equation} where $\{p^*_n\}_{n
\in \N}$ is the sequence of optimal approximants to $1/f$.\end{thm}

It would be desirable to complete the proof of Beurling's Theorem
with a direct proof that any of the known characterizations actually
match the definition of an outer function.

Brown-Shields Theorem on capacity of the zero sets can be indeed
easily proved from the definition of cyclicity in terms of
polynomials, as a corollary to the following result (which can be
found as Theorem 3.3.1 in \cite{EFKMR}):

Denote by $f^*$ the function defined on the boundary by
nontangential limits of $f$, and by $Cap(E)$, the logarithmic
capacity of a set $E$.
\begin{thm}[Weak-type inequality for capacity]
There exists an absolute constant $C$ such that for $f \in D_1$ and
$t>0$ we have \begin{equation}\label{eqn8} Cap(\{|f^*|>t\}) \leq
\frac{C}{t^2} \|f\|^2.
\end{equation}
\end{thm}

From here, to prove Brown and Shields Theorem one only needs to see
that $$Z(f) \subset Z(p^*_n f) \subset \{|p^*_n f -1|>
1-\varepsilon\}.$$

A strengthening of the other result of Brown and Shields (on simple
functions with zeros on the boundary being cyclic) using only the
optimal approximants was already given in \cite{BCLSS13}, and so,
this theory could represent a unified approach to several results on
cyclic functions.

\subsection{Matrices and algorithms}
A Grammian is a matrix given by the inner products of a sequence of
functions with themselves, a matrix with entries
$G_{j,k}=\left<f_j,f_k\right>$. Grammians form a family of matrices
that have been studied for more than a hundred years, for their
relations with the orthogonal projection. The matrix $M$ appearing
in \eqref{eqn6} is a Hermitian Grammian but in some cases it will
have additional structure. Continuing with the case $\alpha=0$, we
may notice that $M_{j,k} = M_{j-k,0}$. A matrix with this property
is called a \emph{Toeplitz matrix}. The Toeplitz structure was
exploited in \cite{BKLSS4} in order to show that we can characterize
cyclicity in Hardy in terms of the zeros of $p^*_n$ exclusively. In
order to show this, it was relevant to study the Levinson algorithm,
which is an efficient algorithm for the inversion of a Toeplitz
matrix. See \cite{Levin}. Toeplitz inversion algorithms are
typically based on either the Schur (see \cite{Schur}) or the
Levinson algorithms.

The reason why the matrices appearing in Hardy are Toeplitz matrices
is that the shift is an isometry in the Hardy space (onto its
image). Unfortunately, this is clearly not true in any other of the
$D_{\alpha}$ spaces: for $\alpha >0$, the shift increases the norm
of a function, whereas for $\alpha<0$, it reduces this norm.
However, the shift in the Dirichlet space ($\alpha=1$) does have a
special property: it is a 2-isometry, meaning that
\begin{equation}\label{eqn7}\|f\|_1^2 - 2 \|Sf\|_1^2 + \|S^2f\|_1^2=0.\end{equation}
In other words,\begin{equation}\label{eqn11} M_{j,k} -
M_{j+1,k+1}=M_{j+1,k+1}-M_{j+2,k+2}.\end{equation}

Therefore we can propose the following problem:

\begin{prob}
Develop an analogous of the Levinson or Schur algorithms that
exploits the structure of a Hermitian matrix which satisfies
\eqref{eqn11}, in order to compute its inverse.
\end{prob}

Ideally, a recursive formula for the inverse matrix could lead to a
similar condition to that in \eqref{eqn10}.

\section{Examples and the role of the logarithmic potential}

\subsection{Brown and Cohn's examples} There are several sources of positive results for the
Brown-Shields Conjecture. For instance, right after Brown and
Shields paper appeared, it was shown in \cite{BC85} that the
conjecture is sharp:

\begin{thm}[Brown-Cohn]
Let $E \subset \T$ be a closed set of logarithmic capacity zero.
Then there exists a cyclic function $f \in D_1$ such that $E \subset
Z(f)$.
\end{thm}
The example functions constructed in this paper satisfy additional
regularity properties: they are functions continuous to the
boundary, and they have logarithms that are also in $D_1$. Although
the proof of cyclicity of these functions is left for the reader,
this is easily derived from a more general later statement by Aleman
(\cite{Aleman}), showing as a particular case that any $f \in D_1$
such that $\log f \in D_1$ must be cyclic. A sufficient condition
for a function $f$ to be outer is $\log f \in H^1$ (where $H^1$
denotes the space of holomorphic functions on the disk, integrable
on the boundary), or equivalently $(\log f)^{1/2} \in D_0$. At the
same time, if $\alpha \in (0,1]$ and there exists a number $t>0$
such that $(\log f)^t \in D_{\alpha}$ then the $\alpha$-capacity of
the zero set of $f$ will satisfy the corresponding necessary
condition for cyclicity, that $Cap_{\alpha}(Z(f)) =0$. Therefore, it
seems natural to ask the following:

\begin{prob}
Fix $\alpha \in (0,1]$. What are the values of $t >0$ such that for
any $f\in D_{\alpha}$, if $(\log f)^t \in D_{\alpha}$ then $f$ is
cyclic in $D_{\alpha}$?
\end{prob}

We know $1/2$ works (and is optimal) for $\alpha =0$ and $1$ does
for $\alpha>0$. It seems natural to think that $1/2$ may work in all
cases. This seems related with another problem that is closely
related with Brown-Shields Conjecture:

\begin{prob}
Fix $\alpha \in (0,1]$. What are the values of $(t,\beta) \in \R^2$
such that for any $f\in D_{\alpha}$, whenever $(\log f)^t \in
D_{\beta}$ and the $\alpha$-capacity of $Z(f) \cap \T$ is zero, then
$f$ is cyclic in $D_{\alpha}$?
\end{prob}

Again, clearly, $t \geq 1/2, \beta=0$ work for $\alpha=0$ and so do
$t \geq 1$ when $\beta=\alpha$. In fact, all the examples of cyclic
functions for Dirichlet we know of satisfy some such condition with
$\beta=1$: cyclic polynomials satisfy it whenever $t<1/2$ or
$\beta<1$, and the Brown-Cohn examples satisfy it with $t=1=\beta$.

\subsection{An unresolved case} Another result in \cite{BS84} states that for outer
functions in $D_2$ their cyclicity in $D_1$ depends only on their
zero sets, so it seems natural to think that the Brown-Shields
conjecture will be true for the particular case of $f \in D_2$,
although this is yet to be shown. It would be enough to show that
Brown and Cohn's result in \cite{BC85} can be improved in terms of
the regularity of the functions, although the zero sets for $D_2$
could form a smaller class.

\begin{prob}\label{prob1}
Does it hold that for any closed subset $E$ of $\T$ of logarithmic
capacity zero such that $E= Z(f_1)$, for some $f_1\in D_2$, there
exists a function $f_2 \in D_2$ that is cyclic in $D_1$ and such
that $E \subset Z(f_2)$.
\end{prob}

It seems reasonable to expect that the $D_2$ condition does help the
function to be cyclic and that the answer to Problem \ref{prob1} is
positive. After several candidates for a level of regularity that
would improve the cyclicity of the function, the question arises of
finding a function for which something can be done but satisfying
none of the unnecessary regularity conditions.

\begin{prob}
Is there an outer function $f \in D_1 \backslash H^{\infty}$
satisfying all of the following:
\begin{enumerate}
\item $Z(f) \cap \T$ has zero logarithmic capacity.
\item For all $t >0$, $(\log f)^t \notin D_1$.
\item The elements of the matrix $M_{j,k} = \left<z^j f, z^k f\right>$ can be computed from existing information.
\end{enumerate}
\end{prob}

The assumption that $f \notin H^{\infty}$ guarantees that $f$ does
not belong to any of the multiplicative algebras ($D_{\alpha}$ with
$\alpha >1$). The last requirement could be replaced by any other
that allows to work towards proving or disproving the cyclicity of
the function.

\subsection{The minimization of logarithmic energies}
A classical problem in analysis is that of finding sets of $n$
points that minimize the energy generated by a given potential with
certain restrictions. In the plane such potentials are usually
related with the logarithmic potential and this is connected with
the problem of determining the orthogonal polynomials for a
particular measure over the unit circle. A good summary of such
situations can be found in \cite{SaTo}.

In \cite{BKLSS4}, the authors show a correspondence between
orthogonal polynomials for some such measures and optimal
approximants for a function in $D_{\alpha}$. In fact, in the Hardy
space, the zeros of orthogonal polynomials are reflections of the
zeros of optimal approximants, and hence it may happen that sets
minimizing energies tied to some logarithmic potentials describe the
zero sets of the optimal approximants. A very ambitious program
could be based on the following problem:
\begin{prob}\label{prob400}
Given $f \in D_{\alpha}$. Determine whether there exists and
describe a potential for an energy which is minimized, for all $n\in
\N$, at the zero set of the optimal approximant $p^*_n$ to $1/f$ in
$D_{\alpha}$.
\end{prob}

A plausible reduction of this problem is that it could be enough to
study only 2 points on the zero set: on one hand, the interaction
between any two points of the zero set will minimize some energy
described by the rest of the points and the function $f$; on the
other, any two zeros $z_0, z_1$ of an optimal approximant $p$ of
degree $n\geq 2$ for a function $f$ determine also the optimal
approximant of degree 2 to the function $pf/(z-z_0)(z-z_1)$. Hence
it could be enough to solve for polynomials of degree 2 for all
functions. This may still be a large problem.

As an illustration of how to find a closed formula for the optimal
approximants to a small collection of functions, we can look at the
functions $f_a=(1-z)^a$, $a \in \N$ and $a \geq 2$, which has a root
of multiplicity $a$ at $z=1$. The optimal approximants to $1/f_a$
may be computed explicitly with a closed formula in the case of the
Hardy space. In the present paper, we denote by $B$ the beta
function, $B(x,y)=\int_{0}^{1}t^{x-1}(1-t)^{y-1}dt$.
\begin{prop}\label{prop100}
Let $a \in \N$. The $\mathrm{n}$th-order optimal approximant to
$1/(1-z)^a$ in $D_0$ is given by
\begin{equation}
p_n(z)=\sum_{k=0}^{n}\left( \binom{k+a-1}{k}
\frac{B(n+a+1,a)}{B(n-k+1,a)}\right)z^k. \label{prop100approx}
\end{equation}
\end{prop}
\begin{proof}
Let us first compute the elements $M_{j,k}$ of the matrix $M$
\eqref{eqn6} associated with $f_a=(1-z)^a$. Since the matrix in
question is Hermitian, we can, without loss of generality, take $j
\geq k$, and since multiplication by $z^k$ is an isometry, we have
that

$$M_{j,k} = \left\langle z^{j-k} (1-z)^a, (1-z)^a \right\rangle.$$

Substituting the Taylor coefficients of $f_a$, we see that
$$M_{j,k} = \sum_{l=0}^a \sum_{s=0}^a \binom{a}{l} \binom{a}{s} (-1)^{l+s} \left\langle z^{l+j-k}, z^s \right\rangle.$$

By the orthogonality of the system of monomials, only the term in
$s=l+j-k$ is non-zero, and in view of basic properties of binomial
coefficients,
$$M_{j,k} = (-1)^{j-k} \sum_{l=0}^{a+k-j} \binom{a}{l} \binom{a}{a-l+k-j}.$$
Now, applying the Chu-Vandermonde identity, we obtain
\begin{equation}\label{eqn300}
M_{j,k} = (-1)^{j-k} \binom{2a}{a+k-j}.
\end{equation}
We can, from now on, take this to be the definition of $M_{j,k}$, extending its domain to all integers $j$ and $k$. This will simplify notation. 

Let $c_{k,n}$ be a proposed solution to the linear system, and, for
fixed $n \in \N$, suppose that $c_{k,n}$ depends on $k$ as a
polynomial of degree less than or equal to $2a-1$. Then we
substitute the values of $M_{j,k}$ from \eqref{eqn300} into the
linear equations \eqref{eqn6}, and we obtain, for $j=0,\ldots,n$,
\begin{equation}\label{eqn301}
\sum_{k=0}^n M_{j,k} c_{k,n} = \sum_{k=0}^n (-1)^{j-k}
\binom{2a}{a+k-j} c_{k,n} =: A_j.
\end{equation}

Suppose, firstly, that $j \in \{a,\ldots,n-a\}$ (and hence, that $n
\geq 2a$). Set $q_{n,a}(s)= (-1)^a c_{j-a+s,n}$, which is a
polynomial in $s$ of the same degree as $c_{k,n}$ in terms of $k$
(that is, less or equal than $2a-1$). Then $A_j$ may be rewritten as
\begin{equation}\label{eqn302}
A_j = \sum_{s=0}^{2a} (-1)^{s} \binom{2a}{s} q_{n,a}(s).
\end{equation}

For any polynomial of degree $2a-1$ or less, the result of
\eqref{eqn302} is equal to $0$ by Newton's theory of finite
differences. Now we know that, for $j=a,\ldots,n-a$, we have
\begin{equation}\label{eqn303}
A_j = 0.
\end{equation}
Define $E:= \{1,\ldots,a-1\} \cup \{n-a+1,\ldots,n\}$, and suppose
that we chose any polynomial on $k$, $c_{k,n}$, of degree less or
equal to $2a-1 = \# E$ such that $c_{k,n}=0$ for all $k \in
\{1-a,\ldots,-1\} \cup \{n+1,\ldots,n+a\}$. Then for all $j \in E$,
\eqref{eqn301} can still be completed, by adding $0$-terms, to the
form \eqref{eqn302}. We have seen that if $c_{k,n}$ is defined as
\begin{equation}\label{eqn304}
c_{k,n} = t_n \left(\prod_{s=1}^{a-1} (k+s) \right)
\left(\prod_{r=1}^{a} (n+r-k) \right),
\end{equation}
for some $t_n \in \C$ depending only on $n$, then \eqref{eqn303}
holds for all $j=1,\ldots,n$.

That is, the system \eqref{eqn6} is satisfied, provided that $A_0
=1$, and we are still free to choose $t_n$. Clearly, since $c_{s,n}
=0$ for $s=1-a,\ldots,-1$, we know that
\begin{equation*}\label{eqn305}
A_0 = \sum_{k=0}^a M_{0,k}c_{k,n} =  \sum_{k=1-a}^a M_{0,k}c_{k,n}.
\end{equation*}

Newton differences tell us that
\begin{equation*}\label{eqn306}
\sum_{k=-a}^a M_{0,k}c_{k,n}=0,
\end{equation*}
and, hence, by the symmetry of the binomial coefficients,
\begin{equation*}\label{eqn307a}
A_0= -M_{0,a}c_{-a,n}= (-1)^{a+1} c_{-a,n}.
\end{equation*}

Therefore, it is enough to choose $t_n$ so that $c_{-a,n}
=(-1)^{a+1}$. Evaluating $c_{-a,n}$ in \eqref{eqn304} gives
\begin{equation*}\label{eqn307b}
c_{-a,n}= t_n (-1)^{a-1} \frac{\Gamma(a)
\Gamma(n+2a+1)}{\Gamma(n+a+1)}.
\end{equation*}
Therefore, choosing
\begin{equation*}\label{eqn308}
t_n = \frac{\Gamma(n+a+1)}{\Gamma(a) \Gamma(n+2a+1)},
\end{equation*}
we have the optimal approximants.

Multiplying all the different factors together and expressing
everything in terms of the gamma function, we find that
\begin{equation}\label{eqn309}
c_{k,n} =
\frac{\Gamma(k+a)\Gamma(n+a+1-k)\Gamma(n+a+1)}{\Gamma(k+1)\Gamma(n-k+1)\Gamma(a)\Gamma(n+2a+1)}.
\end{equation}

A simple expression for the same quantity in terms of binomial
coefficients and the beta function $B$ is
\begin{equation}\label{eqn310}
c_{k,n} = \binom{k+a-1}{k} \frac{B(n+a+1,a)}{B(n-k+1,a)}.
\end{equation}

To see that \eqref{eqn309} and \eqref{eqn310} are equivalent, just
substitute
\[\binom{k+a-1}{k}=\frac{\Gamma(k+a)}{\Gamma(k+1) \Gamma(a)}\quad \textrm{and}\quad B(x,y) = \frac{\Gamma(x),\Gamma(y)}{\Gamma(x+y)}.\]
\end{proof}

The beginning of the previous method can be used in a general
$D_\alpha$ space. For the particular case of the Dirichlet space, we
can go as far as in \eqref{eqn300} and show that the elements
$M_{j,k}$ of the matrix $M$ are given by
$$M_{j,k} = (-1)^{j-k} \binom{2a}{a+k-j} \frac{k+j+a+2}{2}.$$

For the functions $f_a$ in Proposition \ref{prop100}, it is in fact
possible to check explicitly that $p_n(0)$ converges to $1=1/f(0)$,
and although we knew a priori that the function $f_a$ is cyclic,
this method may be of interest in itself. By $a \approx b$ we will
mean there exist universal nonzero constants $C_1$ and $C_2$ such
that $C_1 b \leq a \leq C_2 b$.

\begin{prop}\label{prop2000}
For the functions $f_a=(1-z)^a$, and the optimal approximants
$p^*_n$ of degree $n$ to $1/f$ in $D_0$,
\begin{equation}\label{eqn3000}
\|p^*_nf_a -1\|^2 \approx a^2 / (n+a+1).
\end{equation}
\end{prop}
\begin{proof}
First, from the expression \eqref{eqn309} it is easy to see
that
\begin{equation}\label{eqn400}
p_n(0)= \frac{\Gamma(n+ a+1)^2}{\Gamma(n+1) \Gamma(n+2 a +1)}.
\end{equation}

Now we will use Euler's formula for the gamma function:
$$\Gamma(t) = \frac{1}{t} \prod_{k=1}^{\infty} \frac{(1+\frac{1}{n})^t}{1+\frac{t}{n}}$$

Applied to \eqref{eqn400}, we arrive at
$$p_n(0)= \prod_{k=n+1}^{\infty} \frac{k(k+ 2a)}{(k + a)^2}$$

That is
$$p_n(0)= \prod_{k=n+1}^{\infty} \left(1- \frac{a^2}{(k + a)^2}\right) =: e^{C_n}$$

which tends to 1 as $n$ goes to infinity since $$C_n =
\sum_{k=n+1}^{\infty} \log \left(1- \frac{a^2}{(k+a)^2}\right)
\approx - a^2 \sum_{k=n+1}^{\infty} \frac{1}{(k + a)^2} = -a^2
\sum_{t=n+a+1}^{\infty} t^{-2}.$$

An elementary computation shows then that $C_n$ is comparable to
$-a^2/(n+a+1)$.

Now we apply that $p_nf-1$ is orthogonal to $\mathcal{P}_n f$, to
see that $$d^2_n =\left<p_n f-1, p_n f-1\right> =
\left<1-p_nf,1\right> = 1-p_n(0),$$ and, hence, the distance $d_n$
is approximated (in terms of absolute constants) as
$$d_n^2 \approx 1-e^{-a^2 / (n+a+1)} \approx a^2 / (n+a+1).$$
\end{proof}

With not much work we can solve the quadratic equation $p_2=0$, to
obtain that \begin{equation}\label{eqn501} Z(p_2):=\{z_0,z_1\} =
\{-1\pm i \sqrt{2/a}\}. \end{equation}

Therefore, one can obtain the distances between the zeros
($2\sqrt{2/a}$), the distances between the zeros and the significant
point $z=1$ ($\sqrt{4+2/a}$), or the modulus of the zeros
($\sqrt{1+2/a}$). It would be a starting step to identify a
corresponding family of potentials whose energies are minimized at
these distances. This is an inverse problem from that of identifying
the points of minimal energy, given the functional.

When we take $a=1$, the optimal approximants may be given in more
general spaces than $D_0$ (see \cite{BCLSS13,FrMaSe,BKLSS4}). Other
natural quantities that may influence the description of the
potential are the distances between two zeros of the function for
which we compute the optimal approximants and the multiplicities of
these zeros. Adding a few degrees of generality, we expect the
problem to stay tractable:

\begin{prob}
Find a closed formula for logarithmic potentials with external
fields whose energy is minimized among sets of 2 points by
$Z(p_2^*)$ where $p^*_2$ is the optimal approximant of degree 2 to
$1/f$, and
\begin{equation}\label{eqn8} f(z)= (1-z)^{\beta}
\left[(z-e^{i\theta})(z-e^{-i\theta})\right]^{\gamma}\end{equation}
for $\alpha \in \R$, $\beta, \gamma \geq 0$, $\theta \in (0, \pi]$.
\end{prob}

\subsection{An extremal problem in Bergman spaces}

Zeros of optimal approximants are restricted as to their positions.
The following result was proved in \cite{BKLSS4}:
\begin{thm}\label{thm5}
Let $f \in D_{\alpha}$ not identically zero, $p^*_n$ the
corresponding optimal approximant, and $z_0 \in Z(p^*_n)$. Then
\begin{equation}\label{eqn3000}|z_0| >
\min(1,2^{\alpha/2}).\end{equation} Moreover, 1 is sharp for all
$\alpha \geq 0$ and for all $\alpha <0$, there exists a function $f
\in D_{\alpha}$ such that $z_0 \in \D$.
\end{thm}

The proof is based on the fact that every zero of an optimal
approximant of degree $n$ to some function $1/f$ is the zero of an
optimal approximant of degree $1$ to a different function. This
reduces the problem to approximants of degree 1, to which the
solution of the linear problem \eqref{eqn6} becomes trivial. The
solution $z_0$ is \begin{equation} \label{eqn4000} z_0 =
\frac{\|zf\|^2}{\left<f,zf\right>}.\end{equation} Applying
Cauchy-Schwartz inequality and computing the norm of the shift
operator yields then the result.

Naturally, one can ask what is the sharp constant for each $\alpha
<0$ (these are often called Bergman spaces). A way to deal with this
problem may be to reformulate it in terms of an extremal problem.
The theory of extremal problems in Bergman spaces has been often
fruitful (see, for instance, \cite{BKetal}) and the variety of
techniques may help solve the problem. We will concentrate on the
case $\alpha=-1$. In order to find the sharp constant for Theorem
\ref{thm5}, we would like to find the infimum of the absolute values
of the right-hand side in \eqref{eqn4000}, or equivalently,
\begin{equation}\label{eqn4001}
\sup \frac{|\left<g,zg\right>|}{\|zg\|^2},
\end{equation}
where the supremum is taken over all the functions $g \in D_{-1}$.

Renaming $f=zg/\|zg\|$, we obtain any function in the unit sphere of
$D_{-1}$ with $f(0)=0$. Using the integral expression of the norm of
$D_{-1}$, we arrive to the following problem:
\begin{prob}\label{prob7}
Compute $$\sup \left\{\left|\int_\D \frac{|f(z)|^2}{z}dA(z)\right|:
f(0)=0, \|f\|_{-1}^2 \leq 1\right\}.$$
\end{prob}

By all of the above, the solution should be a number in the interval
$(1,\sqrt{2}]$.

\section{Higher dimensional phenomena}

Several articles have dealt already with cyclicity in more than 1
complex variable. In the case of the bidisk, $\D^2 = \D \times \D$,
Dirichlet-type spaces are usually defined with a product norm:
\begin{defn}\label{defn3} The \emph{Dirichlet-type space} $D_{\alpha}(\D^2)$, of parameter
$\alpha$ over the bidisk is defined as the space of functions $f$ of
two variables that are holomorphic on each variable at each point of
the bidisk, defined by a Taylor series $f(z_1,z_2)= \sum_{j,k \in
\N} a_{j,k} z_1^j z_2^k$ that satisfies
\begin{equation}\label{eqn2000}
\|f\|^2_{\alpha,\D^2}= \sum_{j,k \in \N} |a_{j,k}|^2
((j+1)(k+1))^{\alpha} < \infty.
\end{equation}
\end{defn}
Some problems on cyclicity in this family of spaces have been
tackled. See \cite{BCLSS2, BKKLSS3} and the references therein for
background information on this topic. A difficulty that arises when
increasing the dimension to 2 is the lack of a Fundamental Theorem
of Algebra: the structure of irreducible polynomials in 2 variables
is much richer.

However, the approach in terms of optimal approximants follows the
same principles as in 1 variable: for each finite set of monomials
$X$, one can find the orthogonal projection of the constant function
$1$ onto $Y= f \Span X$ and that will yield the \emph{optimal
approximant within $Y$} to $1/f$. If we choose a sequence of sets
$\{X_n\}$ with $X_n \subset X_{n+1}$ and $\bigcup X_n = \{z_1^j
z_2^k, (j,k) \in \N^2\}$, a function will be cyclic depending only
on the behavior of these optimal approximants.Natural choices for
$X_n$ are the monomials of degree less or equal to $n$ where the
definition of the degree can be taken to be the maximum or the sum
of the degrees on each variable. In algebraic geometry, it is more
often the latter while the first one is commonly used in analysis.
Here we will use the algebraic version.
\begin{defn}
By the \emph{optimal approximant of degree $n$} to $1/f$ we denote
the optimal approximant to $1/f$ within $f\Span \{z_1^j z_2^k :
(j,k)\in \N^2, j+k \leq n\}$.
\end{defn}

Discrete sets of points that minimize some energy are often studied
in higher dimensions (for example, in sampling theory), but
algebraic varieties, of dimension greater or equal to 1, that
minimize a functional are pointing in a completely different
direction. The typical pathologies of minimal currents may occur
only when taking limits.

Let us explore an example. Choose $f(z_1,z_2)= 1-
\frac{z_1+z_2}{2}$. We can compute the optimal approximant of degree
1, $p^*_1$: from the symmetry of the coefficients and the uniqueness
of the orthogonal projection, one can see that $p_1^*$ will be of
the form $p_1^*(z_1,z_2)= a_0(a_1+ (z_1+z_2))$. The constants $a_0$
and $a_1$ will depend on the parameter $\alpha$ of the space, but
$a_1$ can be shown to be a real number larger than 2. In particular,
$Z(p_1^*)$ does not intersect the bidisk.

This example is in concordance with what happens in one dimension,
at least for $\alpha \geq 0$, although an analogous to Theorem
\ref{thm5} is not known yet. A possible restriction could be that
zero sets of optimal approximants can't intersect the bidisk when
$\alpha \geq 0$. But this wouldn't tell which irreducible
polynomials are feasible as optimal approximants (observe this is
answered by Theorem \ref{thm5}). When $\alpha <0$, the question
retains some uncertainty equivalent to solving the Problem
\ref{prob7}. Describing all the irreducible polynomials that appear
as optimal approximants seems a difficult task, but many subproblems
may be of interest. We propose the following:
\begin{prob}\label{prob8}
For each value of $\alpha \geq 0$, determine which algebraic curves
are zero sets of optimal approximants of degree 2 or less to $1/f$,
for some $f \in D_{\alpha}(\D^2)$.
\end{prob}


\begin{thebibliography}{1}

\bibitem{BKetal} D. Aharonov, C. B\'en\'eteau, D. Khavinson, and H. Shapiro, Extremal
Problems for Nonvanishing Functions in Bergman Spaces, Operator
Theory: Advances and Applications {\bf 158}, 59--86, Birkh\"auser
Verlag Basel, Switzerland, 2005.

\bibitem{Aleman} A. Aleman, The multiplication operator on Hilbert
spaces of analytic functions, Habilitationsschrift,
Fernuniversit\"at Hagen, 1993.

\bibitem{BCLSS13} C. B\'en\'eteau, A. A. Condori, C. Liaw, D. Seco, and
A. A. Sola, Cyclicity in Dirichlet-type spaces and extremal
polynomials, J. Anal. Math {\bf 126} (2015), 259--286.

\bibitem{BCLSS2} C. B\'en\'eteau, A. A. Condori, C. Liaw, D. Seco, and
A. A. Sola, Cyclicity in Dirichlet-type spaces and extremal
polynomials II: functions on the bidisk, Pacific J. Math. {\bf 276}
(2015), No.1, 35--58.

\bibitem{BKLSS4} C. B\'en\'eteau, D. Khavinson, C. Liaw, D. Seco, and
A. A. Sola, Orthogonal polynomials, reproducing kernels, and zeros
of optimal approximants, preprint available at
http://arxiv.org/abs/1509.04807.

\bibitem{BKKLSS3} C. B\'en\'eteau, G. Knese, \L. Kosi\'nski, C. Liaw, D. Seco, and
A. A. Sola, Cyclic polynomials in two variables, to appear on Trans.
AMS.

\bibitem{Beur} A. Beurling, On two problems concerning linear
transformations in Hilbert space, Acta Math. {\bf 81} (1949),
239--255.

\bibitem{BC85} L. Brown and W. C. Cohn, Some examples of cyclic
vectors in the Dirichlet spce, Proc. AMS {\bf 95} (1985), 42--46.

\bibitem{BS84} L. Brown and A. L. Shields, Cyclic vectors in the Dirichlet space,
Trans. Amer. Math. Soc. {\bf 285} (1984), 269--304.

\bibitem{Duren} P. L. Duren,
Theory of $H^p$ spaces, Academic Press, New York, 1970.

\bibitem{EFKMR} O. El-Fallah, K. Kellay, J. Mashreghi, and T. Ransford,
A primer on the Dirichlet space, Cambridge Tracts in Math. {\bf
203}, Cambridge University Press, 2014.

\bibitem{FrMaSe} E. Fricain, J. Mashreghi, and D. Seco, Cyclicity in
Reproducing Kernel Hilbert Spaces of analytic functions, Comput.
Methods Funct. Theory {\bf 14} (2014) 665--680.

\bibitem{Garnett} J. B. Garnett,
Bounded Analytic Functions, Graduate Texts in Mathematics,
Springer-Verlag, 2007.

\bibitem{HKZ} H. Hedenmalm, B. Korenblum, and K. Zhu,
Theory of Bergman spaces, Graduate Texts in Mathematics,
Springer-Verlag, New York, 2000.

\bibitem{HS90} H. Hedenmalm and A. L. Shields,
Invariant subspaces in Banach spaces of analytic functions, Michigan
Math. J. {\bf 37} (1990), 91--104.

\bibitem{Levin} N. Levinson, The Wiener RMS error criterion in
filter design and prediction, J. Math. Phys. {\bf 35} (1947)
261--278.

\bibitem{Niko} N. K. Nikol'skii,
Treatise on the shift operator, Grundlehren der mathematischen
Wissenschaften {\bf 273}, Springer-Verlag, 1986.

\bibitem{SaTo} E. B. Saff  and V. Totik, Logarithmic potentials with external fields, Grundlehren der mathematischen Wissenschaften {\bf 316}, Springer-Verlag, 1997.

\bibitem{Schur} I. Schur, \"Uber Potenzreihen, die im Innern des
Einheitskreises beschr\"ankt sind, J. Reine Angew. Math. {\bf 147}
(1917), 205--232.

\end{thebibliography}
\end{document}